\documentclass[12pt,a4paper]{article}
\usepackage{amssymb}
\usepackage{amsmath}
\usepackage{amsthm}
\usepackage{graphicx}
\usepackage[utf8]{inputenc}
\usepackage[T1]{fontenc}
\usepackage[english]{babel}
\usepackage{placeins}
\newcounter{cptRef}
\newcounter{cptTh}
\newtheorem{theorem}[cptTh]{Theorem}

\newtheorem{property}[cptTh]{Property}
\newtheorem{corollary}[cptTh]{Corollary}
\newtheorem{conjecture}[cptTh]{Conjecture}

\newtheorem{definition}[cptTh]{Definition}
\newtheorem{observation}[cptTh]{Observation}

\usepackage[textwidth=20mm,obeyFinal,colorinlistoftodos]{todonotes}

\usepackage{hyperref}

\def\H{{\mathcal H}} 
\def\Z{{\mathbb Z}}
\def\set#1{\left\{#1\right\}}
\def\abs#1{\left|#1\right|}
\def\Cay{{\rm Cay}}

\def\paragraph#1{\par{\bf #1} \ignorespaces}

\begin{document}

\title{Homomorphisms of Cayley graphs\penalty-10{} and Cycle Double Covers}

\author{%
  Radek Hu\v{s}ek\and
  Robert \v{S}\'{a}mal\thanks{
Both authors are members of Computer Science Institute of Charles University,
Prague, Czech Republic.
Both authors were partially supported by GA\v{C}R grant 16-19910S.
The first author was partially supported by the Charles University,
project GA UK No.~1726218.
Email: {\tt \{husek,samal\}@iuuk.mff.cuni.cz}
}}

%\begin{keyword}
%Cayley graph, graph homomorphism, cycle double cover, group flow
%\end{keyword}

\date{}

\maketitle

\begin{abstract}
  We study the following conjecture of Matt DeVos: If there is a graph homomorphism
  from Cayley graph $\Cay(M, B)$ to another Cayley graph $\Cay(M', B')$ then
  every graph with an $(M,B)$-flow has an $(M',B')$-flow. This conjecture was originally
  motivated by the flow-tension duality. We show that a natural strengthening of this
  conjecture does not hold in all cases but we conjecture that it still holds
  for an interesting subclass of them and we prove a partial result in this direction.
  We also show that the original conjecture implies the existence of an oriented cycle
  double cover with a small number of cycles.
\end{abstract}

\section{Introduction}
For an abelian group $M$ (all groups in this article are abelian even though we often
omit the word abelian), an \emph{$M$-flow $\varphi$} on a directed graph $G = (V, E)$
is a mapping $E \to M$ such that the oriented sum around every vertex $v$ is zero:
$$ 
  \sum_{vw \in E} \varphi(vw) - \sum_{uv \in E} \varphi(uv) = 0.
$$
We say that $M$-flow $\varphi$ is an \emph{$(M, B)$-flow} if $\varphi(e) \in B$ for all
$e \in E$ (we always assume that $B \subseteq M$ and that $B$ is symmetric,
i.\,e.\ $B = -B$).

An \emph{$M$-tension $\tau$} is again a mapping $E \to M$ but the condition is that 
the oriented sum along every cycle $C$ is zero, explicitly  
$$ 
  \sum_{e \in C^+} \tau(e) - \sum_{e \in C^-} \tau(e) = 0
$$
where $C^+$ are edges of $C$ with one orientation along the cycle and $C^-$ the edges with the opposite orientation.
We define \emph{$(M, B)$-tension} to be an $M$-tension which uses only values from a symmetric
set $B \subseteq M$.

For planar graphs, flows and tensions are dual notions -- every flow in primal graph
induces a tension in its dual and vice versa. A tension can be equivalently described
by a mapping $p\colon V \to M$ (usually called a group coloring or a potential).
The value $\tau(uv)$ is defined as $p(v) - p(u)$, we write $\tau = \delta p$.
Note that $p$ is nothing more than a homomorphism into Cayley graph of $M$.
Because a composition of homomorphisms is a homomorphism, the following statement
holds:

\begin{observation} \label{obs:tensions}
Let $M$, $M'$ be abelian groups and $B \subset M$,
$B' \subset M'$ their symmetric subsets.
If there is a graph homomorphism from $\Cay(M, B)$ into $\Cay(M', B')$, then
every graph with an $(M,B)$-tension has an $(M',B')$-tension.
\end{observation}

Many questions about flows on graphs were motivated by mimicking the properties of coloring  
in the dual setting~\cite{Tutte}. In the same spirit, we ask for the dual version of
Observation~\ref{obs:tensions}: 

\begin{conjecture}[DeVos \cite{ConDeVos}] \label{con:devos}
Let $M$, $M'$ be abelian groups and $B \subset M$,
$B' \subset M'$ their symmetric subsets.
If there is a graph homomorphism from $\Cay(M, B)$ into $\Cay(M', B')$, then
every graph with an $(M,B)$-flow has an $(M',B')$-flow.
\end{conjecture}

This is still an open problem but it holds in some special cases (the first three
appear in \cite{ConDeVos}, the last one is probably new).
\begin{itemize}
  \item If $G$ is planar (because of duality and Observation~\ref{obs:tensions}).
  \item If $0 \in B'$ (every graph has an $(M', \set{0})$-flow).
  \item If $B = M \setminus \set{0}$ and $B' = M' \setminus \set{0}$:
    Here an $(M,B)$-flow is just a nowhere-zero $M$-flow. It is known that
    the existence of a nowhere-zero flow is monotone in the size of the group \cite{Tutte}.
  \item If $M = \Z_{2n+1}$, $B = \set{n, n+1}$, $M'=\Z_{2n-1}$, $B'= \set{n-1,n}$:
    In this case $\Cay(M,B) \cong C_{2n+1}$ has a graph homomorphism to $\Cay(M',B') \cong C_{2n-1}$,
    we will show how to transform an $(M,B)$-flow into an $(M',B')$.
    Let $f$ be an $(M,B)$-flow on a graph~$G$. 
    It is known~\cite{Tutte2} that $G$ also has an integer flow $f'$ such that for every edge~$e$ we have 
    $|f'(e)| < 2n+1$ and $f(e) \equiv f'(e) \pmod{2n+1}$. This means that $f'(e) \in \{\pm n, \pm (n+1)\}$, in other words 
    $f'$ is a nowhere-zero fractional $\frac{2n+1}{n}$-flow. In the other direction it is easy to see that a graph with a 
    fractional $\frac{2n+1}{n}$-flow also has an $(M,B)$-flow. It is known~\cite{GoddynTarsiZhang} that 
    the existence of a nowhere-zero fractional $\frac{2n+1}{n}$-flow implies 
    the existence of a nowhere-zero fractional $\frac{2n-1}{n-1}$-flow. 
\end{itemize}

\section{New Framework}

The structure of homomorphism from $\Cay(M, B)$ to $\Cay(M', B')$ is hard
to describe. Instead we take any mapping $m\colon M \to M'$ (not necessarily a group
homomorphism)
and let $B'$ be determined by $m$ (so $B'$ is the minimal set for which
$m$ is a graph homomorphism). This is achieved by the following technical definition:

\begin{definition}
Let $M, M'$ be abelian groups and $m\colon M \to M'$ any mapping. For $x \in M$
we define its {\em homomorphic image}
$$ \H_m(x) := \set{ m(a + x) - m(a) : a \in M } . $$
\end{definition}

We omit the index $m$ whenever possible.
Observe that in the case of tensions $\H(x)$ is exactly the set of possible
images of value $x$ on some edge after composing original tension
represented by a group coloring with $m$:

\begin{observation}\label{obs:tension_shp}
  Let $p\colon V \to M$ be a group coloring and
  let $m\colon M \to M'$ be any mapping between abelian groups $M$ and $M'$.
  Define $p' = m \circ p$, $\tau = \delta p$, and $\tau' = \delta p'$.
  Then
  \begin{itemize}
    \item $\tau$ is a $M$-tension, 
    \item $\tau'$ is a $M'$-tension, and
    \item $\forall e \in E : \tau'(e) \in \H(\tau(e)).$
  \end{itemize}
\end{observation}

\begin{conjecture}[Reformulation of Conjecture~\ref{con:devos}]\label{con:devos_ref}
Let $M$ and $M'$ be abelian groups and $B$ a symmetric subset of $M$.
For any mapping $m\colon M \to M'$
every graph with an $(M,B)$-flow has an $(M', \bigcup_{x \in B} \H(x))$-flow.
\end{conjecture}

The traditional approach to solving flow-related conjectures is to study properties
of hypothetical minimal counterexample. Usually the problem is reduced to
cubic graphs by splitting\,/\,decontracting vertices. This, however, is not
possible with Conjecture~\ref{con:devos} because decontracting a vertex may
create an edge with a new value found nowhere else, modifying $B$. To overcome
this we formulated the following property:

\begin{property}[Strong homomorphism property] \label{con:husek}
Let $G$ be a (directed) graph and $M$ an abelian group. We say that $G$ has
{\rm strong homomorphism property} (SHP) for $M$ if for every mapping $m\colon M \to M'$
(where $M'$ is any abelian group) and every $M$-flow $\varphi$ there exists
a $M'$-flow $\varphi'$ such that $\varphi'(e) \in \H(\varphi(e))$
for all edges $e$. We say that $G$ has SHP if it has SHP for all abelian groups.
\end{property}

Note that SHP allows the flow to be zero on some edges but such edges
are not interesting because $\H(0)$ is always $\set{0}$.
The SHP allows us to study only cubic graphs -- we can make any graph
(sub)cubic by decontracting its vertices of high degree and if SHP holds for such
a decontracted graphs then it holds for the original graph too. To state this
in a formal way, we need the following technical definition:

\begin{definition}\label{def:decontraction}
We say that digraph a $H$ is a {\em cubification}
of digraph $G$ if
$H$ can be obtained from $G$ using following operations:
\begin{enumerate}
  \item decontraction of vertex of degree at least 4 (such that both new
    vertices have degree at least 3),
  \item suppression of a vertex of degree 2,
  \item deletion of a bridge,
  \item deletion of a loop, and
  \item deletion of an isolated vertex.
\end{enumerate}
\end{definition}

With this definition we want to show that every non-cubic graph can be reduced
to a smaller cubic one. To get this we need to use a slightly non-standard definition
of the size of the graph which considers graphs with larger degrees bigger.
Suitable definition for us is $$ \Phi := \sum_{v \in V} 3^{\deg v}.$$

\begin{observation}[Reducibility of SHP to cubic graphs]\label{obs:shp_red}
Let $M$ be an abelian group and let $G$ be a digraph.
If some cubification of $G$ has SHP for $M$, then also $G$ has SHP for $M$.
Moreover for every flow $\varphi: E \to M$ there exists a non-strictly smaller
(possibly empty)
cubic graph $G'$ and a nowhere-zero flow $\varphi': E' \to M$ such that
if SHP does not hold for $\varphi$ on $G$ then SHP also does not hold
for $\varphi'$ on $G'$.
\end{observation}
\begin{proof}
To prove the first part, we only need to show that inverse of each operation
used in Definition~\ref{def:decontraction} does not break SHP:
\begin{enumerate}
  \item Suppose $G = G_1/e$ and $G_1$ has SHP for~$M$, let $m: M \to M'$ be any mapping. 
    Let $\varphi$ be an $M$-flow on~$G$. 
    There is a unique extension of~$\varphi$ to $G_1$, we use $\varphi$ for this extension as well. 
    (Note that the value of $\varphi(e)$ may be $0$.) As $G_1$ has SHP for $M$, there is an $M'$-flow 
    $\varphi'$ on~$G_1$ such that $\varphi'(e) \in \H_m(\varphi(e))$. The restriction of~$\varphi'$ to~$G = G_1/e$ 
    is the desired $M'$-flow on~$G$. 
  \item Subdivision of an edge is obvious when the new vertex of degree 2
    has both in-degree and out-degree 1. In the other case
    SHP still holds because $\H(-x) = -\H(x)$.
  \item Addition of a bridge does not break SHP because flow on a bridge
    is always 0 and $\H(0) = \set{0}$.
  \item Addition of a loop is also simple because $\H(x)$ is always non-empty
    and we can assign any value on a loop without affecting the rest of the flow.
  \item Addition of an isolated vertex does not change the flow at all.
\end{enumerate}

The moreover part:
Note that $\Phi = \sum_{v \in V} 3^{\deg v}$ for every cubification is strictly
smaller than $\Phi$ of the original graph.
To obtain $G'$ we set $G' = G$, $\varphi' = \varphi$, and
apply following operations as long as possible:
\begin{enumerate}
  \item Remove an edge $e' \in E'$ such that $\varphi'(e') = 0$.
  \item Apply some cubification operation on $G'$.
\end{enumerate}
Because each of the operations decreases $\Phi(G')$, the process terminates.
If the resulting $\varphi'$ was not nowhere-zero, we still could remove
an edge with 0 flow, and if $G'$ was not cubic, we could get a non-trivial
cubification.
\end{proof}

The SHP is a natural strengthening of Conjecture~\ref{con:devos_ref} -- we just fix
a particular $(M, B)$-flow $\varphi$ and try to find a $(M', B')$-flow
$\varphi'$ with an extra requirement $\varphi'(e) \in \H(\varphi(e))$.
Observation~\ref{obs:tension_shp} shows that a variation of SHP for tensions holds
in general, so also all planar graphs have SHP due to duality.

With computer aid we found out that SHP does not hold in general.
The smallest counterexample we found is $K_{3,3}$ with a particular $\Z_5$-flow
and the universal mapping (see Figure~\ref{fig:countex} and Definition~\ref{dfn:univmap}
below).
Although SHP does not hold for $K_{3,3}$ in general it still holds for groups
$\Z_3$ and $\Z_4$ because SHP always holds for groups of size at most 4 
(Theorem~\ref{thm:small_groups}). We also tested that SHP holds for Petersen graph
and $\Z_5$ (we did not try larger snarks due to computational complexity).
This motivates our next conjecture:

\begin{figure}[hbt]
  \centering
  \includegraphics[width=.4\textwidth]{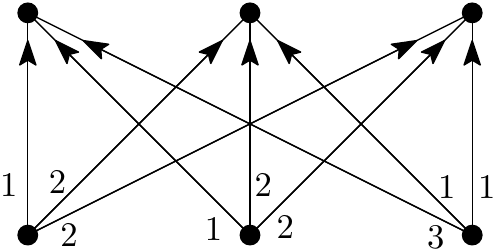}
  \caption{A graph with a $\Z_5$-flow for which SHP does not hold.}
  \label{fig:countex}
\end{figure}

\begin{conjecture}[SHP for minimal groups]
For every graph $G$ the strong homomorphism property holds for group $\Z_k$ where
$k$ is minimal such that $G$ admits a nowhere-zero $\Z_k$-flow.
\end{conjecture}

It is easy to observe that SHP holds for $m$ which are (induced by) a group
homomorphism but a more general statement is true:

\begin{observation} \label{obs:comp}
Let $G$ be a graph and let $m\colon M \to M'$ be some mapping of abelian groups. Let
$h\colon M' \to M''$ be a group homomorphism. If SHP (resp. Conjecture~\ref{con:devos_ref})
holds for $G$ and $m$ then it also holds for $G$ and $h \circ m$.
\end{observation}
\begin{proof}
Let $\varphi'$ be an $M'$-flow guaranteed by SHP. For a group homomorphism $h$ holds
$\H_h(x) = \set{ h(x) }$. Then $\varphi'' = h \circ \varphi'$
is also an $M''$-flow and its values satisfy
\begin{gather*}
\varphi''(e) \in \H_h[ \H_m(\varphi(e)) ] =
\set{ h(m(a - x) - m(a)) : a \in M } = \\  =
\set{ h(m(a - x)) - h(m(a)) : a \in M } = \H_{h \circ m}(x) . \qedhere
\end{gather*}
\end{proof}

\section{Universal objects} 

Observation~\ref{obs:comp} leads us to the definition of a universal mapping such that if SHP
(or Conjecture~\ref{con:devos_ref}) holds for this mapping, it also holds for
every other mapping.

\begin{definition}[Universal mapping] \label{dfn:univmap}
Let $M$ be an abelian\penalty-0{} group. We define its {\em universal group}
$\mathcal G_M = \Z^{M}$ and its {\em universal mapping}
$\mathcal M_M\colon M \to \mathcal G_M$:
$$ x \longmapsto g_x $$
where $g_x$ is a vector with $1$ on position $x$ and $0$ elsewhere.
\end{definition}

The group
$\Z^{M \setminus \set{0}}$ (with 0 mapped to 0 instead of $g_0$)
would be sufficient but we choose the definition with $\Z^M$ to simplify the proofs.
Note that with this definition $\mathcal G_M$ is just a free group generated by elements of $M$.

\begin{observation}
The universal mapping is universal for both Conjecture~\ref{con:devos_ref} and
SHP, i.\,e., if for a given graph (and flow) Conjecture~\ref{con:devos_ref} (resp. SHP)
holds for the universal mapping then it holds for every mapping.
\end{observation}
\begin{proof}
Let $m\colon M \to M'$ be any mapping.
We can also interpret $m$ as a homomorphism
$m_{\rm ext}\colon \mathcal G_M \to M'$
-- mapping $m$ defines values of generators $e_x$
and hence it can be uniquely extended into mapping on the whole group which is
a homomorphism (here we are using the fact that $k g_x = 0 \Rightarrow k = 0$ for
$x \neq 0$). Moreover $m = \mathcal M_M \circ m_{\rm ext}$ so
Observation~\ref{obs:comp} finishes the proof.
\end{proof}

There also exists a universal object on the left-hand side -- a universal flow
$\mathcal F$ (which is just the flow into a free group generated by edges outside
of some fixed spanning tree) -- but $\mathcal G_{\mathcal F}$ has infinitely many
generators so we have not found any reasonable way to work with it. Also
note that although the universal group is infinite, SHP holds for the universal group $\Z^M$
if and only if it holds for $\Z_k^M$ for any $k > \Delta(G)$.

\section{Partial results} 

In this section we prove SHP for some special cases of the mapping or the group. 

\begin{theorem}[Mappings with one ``hole'']
Strong homomorphism property holds for mappings $m\colon \Z_k \to \Z_l$ defined
by $m(x) = a x \bmod l$ where $a \in \Z$. (Here we interpret elements of $\Z_k$
as integers $0, 1, 2, \dots, k - 1$.)
\end{theorem}
\begin{proof}
Note that mapping $m$ is a composition of mappings $m_1\colon \Z_k \to \Z$ defined
by $m_1(x) = x$ and $m_2\colon \Z \to \Z_l$ defined by $m_2(x) = ax \bmod l$, and that 
$m_2$ is a group homomorphism. Hence we only need to show that SHP holds for $m_1$,
the rest follows from Observation~\ref{obs:comp}.

So we need to show that for every $\Z_k$-flow $\varphi$ there exists a $\Z$-flow
$\varphi'$ such that $\varphi'(e) \in \H_{m_1}(\varphi(e)) = \set{ \varphi(e), \varphi(e) - k }$.
This, however, is a well-known result of Tutte~\cite{Tutte2}. 
\end{proof}

\begin{theorem}\label{thm:small_groups}
  Strong homomorphism property holds for groups $\Z_2$, $\Z_3$, $\Z_2^2$, and $\Z_4$.
\end{theorem}
\begin{proof}
Due to Observation~\ref{obs:shp_red} we know that minimal counter-example
is a cubic graph $G$ and nowhere-zero flow $\varphi$. We denote
the generators of the right-hand side free group $a$, $b$, $c$, \dots

\begin{itemize}
\item $\Z_2$: The only graph cubic graph with nowhere-zero $\Z_2$-flow
is the empty graph, for which the claim holds.

\item {$\Z_3$:} Let the mapping $m$ be $0 \mapsto a$, $1 \mapsto b$,
and $2 \mapsto c$ so $\H(1) = \set{b - a, c - b, a - c}$.
A cubic graph has nowhere-zero $\Z_3$-flow if and only if
it is bipartite. So we make all edges directed from one partition to the other
and split them into 3 perfect matchings. Observe that either
$\varphi \equiv 1$ or $\varphi \equiv 2$ in which case we flip the orientation
of edges to get the $\varphi \equiv 1$. We assign one of
the following flow values to each matching: $b - a$, $c - b$, $a - c$.

\item {$\Z_2^2$:} Let the mapping $m$ be $00 \mapsto a$, $01 \mapsto b$,
$10 \mapsto c$, and $11 \mapsto d$. Then
\begin{align*}
  \H(01) &= \set{ \pm(a-b), \pm(c-d) }, \\
  \H(10) &= \set{ \pm(a-c), \pm(b-d) }, \\
  \H(11) &= \set{ \pm(a-d), \pm(b-c) }.
\end{align*}
Let $C_1, C_2, C_3\colon E \to \set{0,\pm 1}$ be a 3-CDC of $G$ defined
(here we slightly abuse notation and define $e \in C \Leftrightarrow
C(e) \neq 0$):
\begin{align*}
  e \in C_1 &\Leftrightarrow \varphi(e) \neq 01 \\
  e \in C_2 &\Leftrightarrow \varphi(e) \neq 10 \\
  e \in C_3 &\Leftrightarrow \varphi(e) \neq 11
\end{align*}
And we define $\psi\colon E \to \Z^4$. Recall that $a = (1,0,0,0) \in \Z^4$, and $b$, $c$, $d$ are defined similarly.
$$ \psi =
  \frac{C_1\! + C_2\! + C_3}{2}a +
  \frac{C_1\! - C_2\! - C_3}{2}b +
  \frac{-C_1\! + C_2\! - C_3}{2}c +
  \frac{-C_1\! - C_2\! + C_3}{2}d
$$
It is easy to check that $\psi$ is a $\Z^4$-flow and $\psi(e) \in \H(\varphi(e))$.

\item {$\Z_4$:} We observe that every vertex (with all incident edges in
same direction) has either values $2, 1, 1$ or $2, 3, 3$. Hence edges with value
2 are a perfect matching. When we remove them we obtain disjoint union of circuits
and we modify orientation of remaining edges so they are directed along circuits.
With this orientation values around every vertex are $1,2,3$ so both
edges with value 1 and edges with value 3 are a perfect matching.

Let $m$ be $0 \mapsto a$, $1 \mapsto b$, $2 \mapsto c$, and $3 \mapsto d$. Then
\begin{align*}
  \H(1) &= \set{ b-a, c-b, d-c, a-d }, \\
  \H(2) &= \set{ \pm(c-a), \pm(d-b) }, \\
  \H(3) &= \set{ d-a, a-b, b-c, c-d }
\end{align*}
Let $C_1, C_2, C_3\colon E \to \set{0,\pm 1}$ be a 3-CDC of $G$ defined:
\begin{align*}
  e \in C_1 &\Leftrightarrow \varphi(e) \neq 1 \\
  e \in C_2 &\Leftrightarrow \varphi(e) \neq 2 \\
  e \in C_3 &\Leftrightarrow \varphi(e) \neq 3
\end{align*}
Observe that we can choose orientation of $C_2$ such that no edge has value $-1$
in $C_2$.
And we define $\psi\colon E \to \Z^4$. Recall that $a = (1,0,0,0) \in \Z^4$, and $b$, $c$, $d$ are defined similarly.
$$ \psi =
  \frac{-C_1\! + C_2\! - C_3}{2}a +
  \frac{C_1\! - C_2\! - C_3}{2}b +
  \frac{C_1\! + C_2\! + C_3}{2}c +
  \frac{-C_1\! - C_2\! + C_3}{2}d
$$
It is easy to check that $\psi$ is a $\Z^4$-flow and $\psi(e) \in \H(\varphi(e))$.
\qedhere
\end{itemize}
\end{proof}

\section{Connection to CDC}

Flows obtained from SHP or
Conjecture~\ref{con:devos_ref} can be easily transformed into an oriented cycle
double cover. Cycle Double Cover Conjecture (CDC, \cite{sze73,sey80})
is together with 3-flow, 4-flow and 5-flow conjectures
one of the major open questions in the field of group flows.
Moreover if $G$ has SHP for $M$ and a nowhere-zero $M$-flow then
the obtained CDC is orientable and has only $\abs{M}$ cycles.
This increases importance of Conjecture~\ref{con:devos_ref} and of determining for
which graphs and groups does SHP hold.

\begin{theorem}[Universal group and CDC]
Let $M$ be any abelian group.
If a graph $G$ has a flow in $\mathcal G_M$ using only
values $\bigcup_{x \in M \setminus \set{0}} \H_{\mathcal M_M}(x)$
then it has an orientable cycle double cover using $\abs{M}$ cycles.
\end{theorem}
\begin{proof}
Denote $H = \bigcup_{x \in M \setminus \set{0}} \H(x)$.
Observe that all elements of $H$ are of form $g_a - g_b$ for some $a,b \in M$ and
those $a,b$ are unique.
Fix an $H$-flow $\varphi$. We define directed cycles
(as mappings $E \to \set{-1, 0, 1}$)
$ \mathcal C_x(a) := (\varphi(a))_x $ and claim that
$C = \set{\mathcal C_x}_{x \in M}$ is
an orientable cycle double cover. From definition $C$ covers each edge twice,
once in each direction, and every $\mathcal C_x$ is a flow with values
$\set{-1,0,1}$ because
it is a composition of a flow and group homomorphism so it is a cycle.
\end{proof}

Seymour in 1981~\cite{sey81} proved that every graph without a bridge
admits nowhere-zero $\Z_6$-flow which combined with the previous theorem for $\Z_6$
gives us the following corollary.

\begin{corollary}
Conjecture~\ref{con:devos} implies that every bridgeless graph has an orientable
cycle double cover with at most 6 cycles.
\end{corollary}

\bibliographystyle{abbrv}
\bibliography{literatura}

\begin{thebibliography}{1}

\bibitem{ConDeVos}
M.~DeVos.
\newblock A homomorphism problem for flows.
\newblock {\em Open Problem Garden}.
\newblock \\
  \url{http://www.openproblemgarden.org/op/a_homomorphism_problem_for_flows},
  [retrieved 2017-03-15].

\bibitem{GoddynTarsiZhang}
L.~A. Goddyn, M.~Tarsi, and C.-Q. Zhang.
\newblock On {$(k,d)$}-colorings and fractional nowhere-zero flows.
\newblock {\em J. Graph Theory}, 28(3):155--161, 1998.

\bibitem{sey80}
P.~D. Seymour.
\newblock Disjoint paths in graphs.
\newblock {\em Discrete Mathematics}, 29(3):293--309, 1980.

\bibitem{sey81}
P.~D. Seymour.
\newblock Nowhere-zero 6-flows.
\newblock {\em Journal of Combinatorial Theory, Series B}, 30(2):130 -- 135,
  1981.

\bibitem{sze73}
G.~Szekeres.
\newblock Polyhedral decompositions of cubic graphs.
\newblock {\em Bulletin of the Australian Mathematical Society},
  8(3):367–387, 1973.

\bibitem{Tutte2}
W.~T. Tutte.
\newblock On the imbedding of linear graphs in surfaces.
\newblock {\em Proceedings of the London Mathematical Society},
  s2-51(1):474--483, 1949.

\bibitem{Tutte}
W.~T. Tutte.
\newblock A contribution to the theory of chromatic polynomials.
\newblock {\em Canadian J. Math.}, 6:80--91, 1954.

\end{thebibliography}

\end{document}